\documentclass[11pt,a4j]{article}
\usepackage{amsthm}

    \newtheorem{thm}{Theorem}[section]
    \newtheorem{prop}[thm]{Proposition}
    \newtheorem{lem}[thm]{Lemma}
    \newtheorem{cor}[thm]{Corollary}
    
  \theoremstyle{definition}
    
    \newtheorem{defi}[thm]{Definition}
  \theoremstyle{remark}
    \newtheorem{rem}[thm]{Remark}
    \newtheorem{ex}[thm]{Example}

\usepackage{amsmath}
\usepackage{amssymb}
\usepackage[all]{xy}
\usepackage{graphicx}
\usepackage{mathrsfs}
\usepackage{indentfirst}

\title{Condensed Sets on Compact Hausdorff Spaces}
\author{Koji Yamazaki}
\begin{document}
\maketitle
\begin{abstract}
A condensed set is a sheaf on the site of Stone spaces and continuous maps.
We prove that condensed sets are equivalent to sheaves on the site of compact Hausdorff spaces and continuous maps.
As an application, we show that there exists a model structure on the category of condensed sets.
\end{abstract}
\setcounter{section}{-1}
	\section{Introduction}
A {\it condensed set}, proposed by Dustin Clausen to Peter Scholze in 2018 (cf. \cite{cele100UMI800UP}), is a sheaf on the site {\bf Stone}, where {\bf Stone} is the category of Stone spaces and continuous maps.
The condensed set provides a framework for dealing algebraically with the structure of topological rings/modules/groups (cf, \cite{scholze2019lectures}).
This is a requirement, for example, when analytic geometry is dealt with in algebraic geometry (cf. \cite{clausen2022condensed}). \\[10pt]
We denote the category of compact Hausdorff spaces and continuous maps by {\bf CH}.
{\bf CH} is also a site (cf. Section 1.2).
The site {\bf Stone} is ``cofinal'' in the site {\bf CH} (cf. Proposition \ref{PropStoneProjection}).
Then, the sheaves defined on each sites are equivalent.
The main theorem of this paper is as the following.
\setcounter{section}{2}
\setcounter{thm}{0}
\begin{thm}
For any condensed set $X$, there exists a sheaf $Y$ on the site {\bf CH} such that the restriction of $Y$ to {\bf Stone} is isomorphic to $X$.
Moreover, $Y$ is unique up to isomorphic.
\end{thm}
As an application, we show that there exists a model structure on the category of condensed sets (cf. Section 3 and Corollary \ref{CorModel}).
Theorem \ref{main1} is necessary to prove Lemma \ref{LemSmall}.\\[10pt]
To discuss the {\it sheaf theoritic h-principle}, Gromov \cite{gromov2013partial} uses a {\it quasitopological space}.
Corollary \ref{CorModel} implies that a quasitopological space can be replaced by a condensed set in the {\it sheaf theoritic h-principle} (cf. \cite{yamazaki2021fibration}).
To apply it to the {\it Oka-Grauert principle}, it is necessary to define the completeness and the denseness (cf. \cite{studer2020homotopy}).
		\subsection*{Acknowledgements}
I would like to thank Tomohiro Asano for telling me about a condensed set.
\setcounter{section}{0}
\setcounter{thm}{0}
	\section{Condensed sets}
We will review the condensed sets.
A {\it condensed set} is a sheaf on the site {\bf Stone}, where {\bf Stone} is the category of Stone spaces and continuous maps.
The Grothendieck (pre)topology on {\bf Stone} is defined in Section 1.2.
	\subsection{Stone spaces}
First, we will review the Stone spaces.
\begin{defi} \label{DefStone}
A {\it Stone space} or a {\it profinite set} is a totally disconnected, compact and Hausdorff topological space.
We denote the category of Stone spaces and continuous maps by {\bf Stone}.
\end{defi}
Recall that a topological space $X$ is {\it totally disconnected} if all of connected components are singletons.
Stone spaces correspond to Boolean rings by {\it Stone's representation theorem}.
By this correspondence, complete Boolean rings corresponds to Stonean spaces.
\begin{defi} \label{DefStonean}
A topological space $X$ is {\it extremally disconnected} if the closure of any open subset is open.
A {\it Stonean space} is an extremally disconnected, compact and Hausdorff topological space.
We denote the category of Stonean spaces and continuous maps by {\bf Stonean}.
\end{defi}
We denote the category of compact Hausdorff spaces and continuous maps by {\bf CH}.
The following Proposition means the ``cofinality'' of {\bf Stone} and {\bf Stonean} in {\bf CH}.
\begin{prop} \label{PropStoneProjection}
There exist a functor ${\bf CH} \rightarrow {\bf Stonean}; X \mapsto \widehat{X}$ and a natural transformation $\widehat{X} \rightarrow X$ such that each component is surjective.
\end{prop}
\begin{proof}
The forgetful functor $G: {\bf CH} \rightarrow {\bf Set}$ has a left adjoint functor $F: {\bf Set} \rightarrow {\bf CH}$.
The functor $F$ maps a set $X$ to the Stone–\v{C}ech compactification of $X$ with the discrete topology.
(See \cite{mac2013categories}.)
As is well known, the Stone–\v{C}ech compactification of a discrete space is a Stonean space.
Define $\widehat{X} = FG(X)$ for each compact Hausdorff space $X$.
The counit of the adjunction $F \dashv G$ gives a natural transformation $\widehat{X} \rightarrow X$ such that each component is surjective.
\end{proof}
	\subsection{Topologies on each categories}
Recall that {\bf CH} (resp. {\bf Stone}, {\bf Stonean}) is the category of compact Hausdorff spaces (resp. Stone spaces, Stonean spaces) (cf. Definition \ref{DefStone} and Definition \ref{DefStonean}).
\begin{defi} \label{DefCovering}
Let $X$ be a compact Hausdorff space.
A class of continuous maps $\{ U_i \rightarrow X \}_{i \in I}$ is a loose (resp. middle, tight) covering if the set of index $I$ is finite, $\coprod_i U_i \rightarrow U$ is surjective and all $U_i$ are compact Hausdorff spaces (resp. Stone spaces, Stonean spaces).
\end{defi}
{\bf CH} (resp. {\bf Stone}, {\bf Stonean}) is a site with the loose (resp. middle, tight) coverings as the coverings.
Each other coverings on the categories {\bf CH} and {\bf Stone} does not define a pretopology, (at least not in a popular way).
This is because, for example, the pull-back of a middle covering may not be a middle covering in {\bf CH}.
However, the sheaf conditions they define are all equivalent.
\begin{prop} \label{PropCHCovering}
Let $\mathcal{F}$ be a presheaf on {\bf CH}.
The followings are equivalents.
\begin{enumerate}
\item $\mathcal{F}$ is a sheaf.
\item For any middle covering $\{ U_i \rightarrow U \}_i$ of any compact Hausdorff space $U$, the following diagram is an equalizer:
\[
\mathcal{F}(U) \rightarrow \prod_i \mathcal{F}(U_i) \rightrightarrows \prod_{i, j} \mathcal{F}(U_i \times_U U_j).
\]
\item For any tight covering $\{ U_i \rightarrow U \}_i$ of any compact Hausdorff space $U$, the following diagram is an equalizer:
\[
\mathcal{F}(U) \rightarrow \prod_i \mathcal{F}(U_i) \rightrightarrows \prod_{i, j} \mathcal{F}(U_i \times_U U_j).
\]
\end{enumerate}
\end{prop}
\begin{proof}
{\it 1.} $\Rightarrow$ {\it 2.} $\Rightarrow$ {\it 3.} is trivial.
Assume that we have {\it 3.}.
Take any compact Hausdorff space $U$ and any loose covering $\{ U_i \rightarrow U \}_i$ of $U$.
For each index $i$, take a surjection $\widehat{U_i} \rightarrow U_i$ from a Stonean space $\widehat{U_i}$ to $U_i$ (cf. Proposition \ref{PropStoneProjection}).
Consider the following diagram:
\[\xymatrix{
		\mathcal{F}(U) \ar[r]
		& \prod_i \mathcal{F}(\widehat{U_i}) \ar@<2pt>[r] \ar@<-2pt>[r]
		& \prod_{i, j} \mathcal{F}(\widehat{U_i} \times_U \widehat{U_j})
	\\
		\mathcal{F}(U) \ar[r] \ar@{=}[u]
		& \prod_i \mathcal{F}(U_i) \ar@<2pt>[r] \ar@<-2pt>[r] \ar[u]
		& \prod_{i, j} \mathcal{F}(U_i \times_U U_j). \ar[u]
}\]
The upper row is an equalizer because $\{ \widehat{U_i} \rightarrow U \}_i$ is a tight covering of $U$.
The middle vertical map $\prod_i \mathcal{F}(U_i) \rightarrow \prod_i \mathcal{F}(\widehat{U_i})$ is injective because $\{ \widehat{U_i} \rightarrow U_i \}$ is a tight covering of $U_i$ for any $i$.
Then, the lower row is an equalizer.
Therefore, $\mathcal{F}$ is a sheaf.
\end{proof}
In exactly the same way, we obtain the following proposition.
\begin{prop} \label{PropStoneCovering}
Let $\mathcal{F}$ be a presheaf on {\bf Stone}.
The followings are equivalents.
\begin{enumerate}
\item $\mathcal{F}$ is a sheaf.
\item For any tight covering $\{ U_i \rightarrow U \}_i$ of any Stone space $U$, the following diagram is an equalizer:
\[
\mathcal{F}(U) \rightarrow \prod_i \mathcal{F}(U_i) \rightrightarrows \prod_{i, j} \mathcal{F}(U_i \times_U U_j).
\]
\end{enumerate}
\end{prop}
	\subsection{Condensed sets}
We define a condensed set.
\begin{defi} \label{DefCondensed}
A {\it condensed set} is a sheaf on the site {\bf Stone} (cf. Section 1.2).
\end{defi}
\begin{rem}
The site {\bf Stone} is not small.
Then, the graph of a condensed set is not small.
This is inconvenient when defining a category {\bf Cond} of condensed sets.
It is known that such set-theoretic problems can be avoided (cf. \cite{clausen2022condensed}).
Furthermore, the category {\bf Cond} of condensed sets can be locally small.
\end{rem}
\begin{ex} \label{ExTopCond}
Let $X$ be a topological space.
$X$ has an associated condensed set $\overline{X}$.
$\overline{X}$ is defined as the following: for each Stone space $A$, 
\[
\overline{X}(A) = \{ \mbox{continuous maps} A \rightarrow X \}.
\]
The construction $X \mapsto \overline{X}$ induces a functor $G_1: {\bf Top} \rightarrow {\bf Cond}$, where {\bf Top} is the category of topological spaces and continuous maps.
\end{ex}
The above functor $G_1$ has a left adjoint $F_1: {\bf Cond} \rightarrow {\bf Top}$.
We will define the functor $F_1$.
Each condensed set $X$ has an underlying set $X({\bf 1})$, where {\bf 1} is the singleton.
Fix any Stone space $A$.
For each point $a \in A$, we write the restriction map associated with the constant map ${\bf 1} \overset{a}{\rightarrow} A$ as $a^\ast: X(A) \rightarrow X({\bf 1})$.
Each element $f \in X(A)$ has an underlying map $\underline{f}: A \rightarrow X({\bf 1})$ defined as $\underline{f}(a) = a^\ast(f)$.
Define a topology on the set $X({\bf 1})$ as the strongest topology for which $\underline{f}$ is continuous for any Stone space $A$ and any element $f \in X(A)$.
We write the topological space $X({\bf 1})$ as $F_1(X)$.
The construction $X \mapsto F_1(X)$ induces a functor $F_1: {\bf Cond} \rightarrow {\bf Top}$.
This functor $F_1$ is a left adjoint of the functor $G_1$.
	\section{Main theorem}
We prove the following main theorem.
\begin{thm} \label{main1}
For any condensed set $X$, there exists a sheaf $Y$ on the site {\bf CH} such that the restriction of $Y$ to {\bf Stone} is isomorphic to $X$.
Moreover, $Y$ is unique up to isomorphic.
\end{thm}
This theorem follows immediately from the following theorem.
\begin{thm} \label{main2}
For any sheaf $\mathcal{F}$ on the site {\bf Stonean}, there exists a sheaf $\mathcal{G}$ on the site {\bf CH} (resp. {\bf Stone}) such that the restriction of $\mathcal{G}$ to {\bf Stonean} is isomorphic to $\mathcal{F}$.
Moreover, $\mathcal{G}$ is unique up to isomorphic.
\end{thm}
\begin{proof}[proof of Theorem \ref{main1}]
Take any condensed set $X$.
$X$ is a sheaf on the site {\bf Stone}.
Let $\mathcal{F}$ be the restriction of the sheaf $X$ to {\bf Stonean}.
There exists a sheaf $Y$ on the site {\bf CH} such that the restriction of $Y$ to {\bf Stonean} is isomorphic to $\mathcal{F}$.
Let $X'$ be the restriction of the sheaf $Y$ to {\bf Stone}.
$X'$ is isomorphic to $X$ because the restrictions of $X$ and $X'$ are both $\mathcal{F}$.
Such $Y$ is unique up to isomorphic.
\end{proof}
\begin{proof}[proof of Theorem \ref{main2}]
\mbox{}\\
\underline{\bf uniqueness}\par
Assume that there exist two sheaves $\mathcal{G}$ and $\mathcal{H}$ on the site {\bf CH} (resp. {\bf Stone}) such that the restrictions of them to {\bf Stonean} is isomorphic to $\mathcal{F}$.
It is sufficient if a natural isomorphism $\phi: \mathcal{G} \rightarrow \mathcal{H}$ is constructed.
Take any compact Hausdorff space (resp. Stone space) $X$.
There exists a surjection $\widehat{X} \rightarrow X$ from a Stonean space $\widehat{X}$ (cf. Proposition \ref{PropStoneProjection}).
$\widehat{X} \times_X \widehat{X}$ is also a Stonean space.
A bijection $\phi_X$ is defined by the following diagram:
\[\xymatrix{
		\mathcal{G}(X) \ar[r] \ar@{.>}[d]^-{\phi_X}
		& \mathcal{G}(\widehat{X}) \ar@<2pt>[r] \ar@<-2pt>[r] \ar[d]^-{\cong}
		& \mathcal{G}(\widehat{X} \times_X \widehat{X}) \ar[d]^-{\cong}
	\\
		\mathcal{H}(X) \ar[r]
		& \mathcal{H}(\widehat{X}) \ar@<2pt>[r] \ar@<-2pt>[r]
		& \mathcal{H}(\widehat{X} \times_X \widehat{X}).
}\]
We will show that the map $\phi_X$ does not depend on the surjection $\widehat{X} \rightarrow X$.
We write $\phi_{\widehat{X} \rightarrow X}$ for $\phi_X$ defined by $\widehat{X} \rightarrow X$.
Suppose another surjection $\widehat{X'} \rightarrow X$ from a Stonean space $\widehat{X'}$ exists.
Let $\widehat{X''} = \widehat{X} \times_X \widehat{X'}$.
By the following diagram, we obtain the equality $\phi_{\widehat{X} \rightarrow X} = \phi_{\widehat{X''} \rightarrow X}$:
\[\xymatrix{
		& \mathcal{G}(X) \ar[rr] \ar@{.>}[ddl]_-{\phi_{\widehat{X} \rightarrow X}} \ar@{=}[d]
		&& \mathcal{G}(\widehat{X}) \ar@<2pt>[r] \ar@<-2pt>[r] \ar[ddl]|\hole \ar[d]
		& \mathcal{G}(\widehat{X} \times_X \widehat{X}) \ar[d] \ar[ddl]|(.48)\hole|(.52)\hole
	\\
		& \mathcal{G}(X) \ar[rr] \ar@{.>}[ddl]^(.20){\phi_{\widehat{X''} \rightarrow X}}
		&& \mathcal{G}(\widehat{X''}) \ar@<2pt>[r] \ar@<-2pt>[r] \ar[ddl]
		& \mathcal{G}(\widehat{X''} \times_X \widehat{X''}) \ar[ddl]
	\\
		\mathcal{H}(X) \ar[rr]|(.25)\hole \ar@{=}[d]
		&& \mathcal{H}(\widehat{X}) \ar@<2pt>[r]|\hole \ar@<-2pt>[r]|(.48)\hole \ar[d]
		& \mathcal{H}(\widehat{X} \times_X \widehat{X}) \ar[d]
	\\
		\mathcal{H}(X) \ar[rr]
		&& \mathcal{H}(\widehat{X''}) \ar@<2pt>[r] \ar@<-2pt>[r]
		& \mathcal{H}(\widehat{X''} \times_X \widehat{X''}).
}\]
In exactly the same way, we obtain the equality $\phi_{\widehat{X'} \rightarrow X} = \phi_{\widehat{X''} \rightarrow X}$.
Then, we obtain the equality $\phi_{\widehat{X} \rightarrow X} = \phi_{\widehat{X''} \rightarrow X} = \phi_{\widehat{X'} \rightarrow X}$.
This shows that the map $\phi_X$ does not depend on the surjection $\widehat{X} \rightarrow X$.\par
Next, we will show that $X \mapsto \phi_X$ is natural.
Take any continuous map $X \rightarrow Y$ between compact Hausdorff spaces (resp. Stone spaces).
There exists a surjection $\widehat{Y} \rightarrow Y$ from a Stonean space $\widehat{Y}$.
$X \times_Y \widehat{Y}$ is compact Hausdorff, but may not be a Stonean space.
There exists a surjection $\widehat{X} \rightarrow X \times_Y \widehat{Y}$ from a Stonean space $\widehat{X}$.
We have a surjection $\widehat{X} \rightarrow X \times_Y \widehat{Y} \rightarrow X$.
By the following diagram, $X \mapsto \phi_X$ is natural:
\[\xymatrix{
		& \mathcal{G}(Y) \ar[rr] \ar@{.>}[ddl]_-{\phi_Y} \ar[d]
		&& \mathcal{G}(\widehat{Y}) \ar@<2pt>[r] \ar@<-2pt>[r] \ar[ddl]|\hole \ar[d]
		& \mathcal{G}(\widehat{Y} \times_Y \widehat{Y}) \ar[d] \ar[ddl]|(.48)\hole|(.52)\hole
	\\
		& \mathcal{G}(X) \ar[rr] \ar@{.>}[ddl]^(.20){\phi_X}
		&& \mathcal{G}(\widehat{X}) \ar@<2pt>[r] \ar@<-2pt>[r] \ar[ddl]
		& \mathcal{G}(\widehat{X} \times_X \widehat{X}) \ar[ddl]
	\\
		\mathcal{H}(Y) \ar[rr]|(.25)\hole \ar[d]
		&& \mathcal{H}(\widehat{Y}) \ar@<2pt>[r]|\hole \ar@<-2pt>[r]|(.48)\hole \ar[d]
		& \mathcal{H}(\widehat{Y} \times_Y \widehat{Y}) \ar[d]
	\\
		\mathcal{H}(X) \ar[rr]
		&& \mathcal{H}(\widehat{X}) \ar@<2pt>[r] \ar@<-2pt>[r]
		& \mathcal{H}(\widehat{X} \times_X \widehat{X}).
}\]\par
We have therefore obtained a natural isomorphism $\phi: \mathcal{G} \rightarrow \mathcal{H}$.
This implies uniqueness of $\mathcal{G}$.\\
\underline{\bf existence}\par
Let $\mathcal{F}$ be any sheaf on the site {\bf Stonean}.
We will construct a sheaf $\mathcal{G}$ on the site {\bf CH} (resp. {\bf Stone}) such that the restriction of $\mathcal{G}$ to {\bf Stonean} is isomorphic to $\mathcal{F}$.
Take any compact Hausdorff space (resp. Stone space) $X$.
By Proposition \ref{PropStoneProjection}, there exists a surjection $\widehat{X} \rightarrow X$ from a Stonean space $\widehat{X}$ in functorially.
(To avoid the {\it axiom of choice} for not small sets, note that this construction is functorial.)
Define the set $\mathcal{G}(X)$ as that for which the following diagram is an equalizer:
\[
\mathcal{G}(X) \rightarrow \mathcal{F}(\widehat{X}) \rightrightarrows \mathcal{F}(\widehat{X} \times_X \widehat{X}).
\]
Take any continuous map $X \rightarrow Y$ between compact Hausdorff spaces (resp. Stone spaces).
\[\xymatrix{
		\widehat{X} \ar@{->>}[d] \ar[r]
		& \widehat{Y} \ar@{->>}[d]
	\\
		X \ar[r]
		& Y
}\]
A map $\mathcal{G}(Y) \rightarrow \mathcal{G}(X)$ is induced by the following diagram:
\[\xymatrix{
		\mathcal{G}(X) \ar[r]
		& \mathcal{F}(\widehat{X}) \ar@<2pt>[r] \ar@<-2pt>[r]
		& \mathcal{F}(\widehat{X} \times_X \widehat{X})
	\\
		\mathcal{G}(Y) \ar[r] \ar@{.>}[u]
		& \mathcal{F}(\widehat{Y}) \ar@<2pt>[r] \ar@<-2pt>[r] \ar[u]
		& \mathcal{F}(\widehat{Y} \times_Y \widehat{Y}). \ar[u]
}\]
By the above, a presheaf $\mathcal{G}$ can be defined. \par
We will show that the presheaf $\mathcal{G}$ is a sheaf on the site {\bf CH} (resp. {\bf Stone}).
Take any compact Hausdorff space (resp. Stone space) $U$ and any tight covering $\{ U_i \rightarrow U \}_i$ (cf. Proposition \ref{PropCHCovering} and Proposition \ref{PropStoneCovering}).
Let $U_{ij} = U_i \times_U U_j$.
\[\xymatrix{
		\mathcal{F}(\widehat{U} \times_U \widehat{U}) \ar[r]
		& \prod_i \mathcal{F}(\widehat{U_i} \times_{U_i} \widehat{U_i}) \ar@<2pt>[r] \ar@<-2pt>[r]
		& \prod_{i, j} \mathcal{F}(\widehat{U_{ij}} \times_{U_{ij}} \widehat{U_{ij}})
	\\
		\mathcal{F}(\widehat{U}) \ar[r] \ar@<2pt>[u] \ar@<-2pt>[u]
		& \prod_i \mathcal{F}(\widehat{U_i}) \ar@<2pt>[r] \ar@<-2pt>[r] \ar@<2pt>[u] \ar@<-2pt>[u]
		& \prod_{i, j} \mathcal{F}(\widehat{U_{ij}}) \ar@<2pt>[u] \ar@<-2pt>[u]
	\\
		\mathcal{G}(U) \ar[r] \ar[u]
		& \prod_i \mathcal{G}(U_i) \ar@<2pt>[r] \ar@<-2pt>[r] \ar[u]
		& \prod_{i, j} \mathcal{G}(U_{ij}). \ar[u]
}\]
All columns are equalizers.
The upper row and the middle row are equalizers.
Then, the lower row is an equalizer.
Therefore, the presheaf $\mathcal{G}$ is a sheaf on the site {\bf CH} (resp. {\bf Stone}) by Proposition \ref{PropCHCovering} (resp. Proposition \ref{PropStoneCovering}). \par
Finally, we will show that the restriction of $\mathcal{G}$ to {\bf Stonean} is isomorphic to $\mathcal{F}$.
For any Stonean space $X$, the following diagram is an equalizer:
\[
\mathcal{F}(X) \rightarrow \mathcal{F}(\widehat{X}) \rightrightarrows \mathcal{F}(\widehat{X} \times_X \widehat{X}).
\]
Then, we have a natural bijection $\mathcal{F}(X) \cong \mathcal{G}(X)$ because of the definition of $\mathcal{G}(X)$.
This completes the proof.
\end{proof}
	\section{A model structure on {\bf Cond}}
This section presents the application of the main theorem.
Consider the following adjunctions sequence:
\begin{equation}
\vcenter{\xymatrix{
		{\bf sSet} \ar@<-10pt>[r]_{F_0} \ar@{}[r] | {\top}
		& {\bf Cond} \ar@<-10pt>[l]_{G_0} \ar@<-10pt>[r]_{F_1} \ar@{}[r] | {\top}
		& {\bf Top} \ar@<-10pt>[l]_{G_1}.
}}
\end{equation}
The above categories and functors are defined as follows.
\begin{itemize}
\item {\bf sSet} is the category of simplicial sets.
\item {\bf Cond} is the category of condensed sets.
\item {\bf Top} is the category of topological spaces and continuous maps.
\item $G_0$ is the functor defined as $G_0(X) = {\bf Cond}(-, X)$.
\item $G_1$ is the functor defined as $G_1(X) = \overline{X}$ (cf. Example \ref{ExTopCond}).
\item $F_0$ is defined in the same way as the geometric realizations.
Specifically, this is defined as the following:
\[
F_0(X) := \int^n {\bf sSet}(\Delta^{n-1}, X) \bullet \Delta^{n-1},
\]
where $(-) \bullet (-)$ is copower.
\item $F_1$ is defined as the following (cf. section 1.3):
\[
F_1(X) \mbox{ is the topological space } X({\bf 1}) \mbox{ with the strong topology.}
\]
\end{itemize}
In this section, we will show that the category {\bf Cond} has a model structure such that the above adjunctions $F_0 \dashv G_0$ and $F_1 \dashv G_1$ are Quillen equivalences.
We can transfer the {\it cofibrantly generated} model structure on the category {\bf sSet} along the adjunction $F_0 \dashv G_0$ (cf. \cite{hovey2007model}).
Define a set $I$ (resp. $J$) generating cofibrations (resp. acyclic cofibrations) on {\bf Cond} as the following:
\[
\begin{array}{lcl}
I & = & \{ \partial \Delta^n \hookrightarrow \Delta^n \} \\
J & = & \{ \Delta^n \overset{0}{\hookrightarrow} \Delta^n \times \Delta^1 \}.
\end{array}
\]
Remark that we abuse the symbol to write the condensed set $\overline{\Delta^n}$ associated with a simplex $\Delta^n$ as $\Delta^n$.
The following two claims that need to be proved are.
\begin{description}
\item[(1)] $I$ and $J$ permit the small object argument.
\item[(2)] Any relative $J$-cell complex is a weak homotopy equivalence.
\end{description}\par
	\subsection{Small objects}
We define a small object.
\begin{defi}
Let $\mathscr{C}$ be a cocomplete category.
An object $W$ in $\mathcal{C}$ is {\it $\aleph_0$-small} if, for any regular cardinal $\lambda$ and any $\lambda$-sequence $X$ in $\mathscr{C}$, the following map is a bijection:
\[
\displaystyle\lim_{\substack{\longrightarrow \\ \beta}} \mathscr{C}(W, X_\beta) \cong \mathscr{C}(W, \displaystyle\lim_{\substack{\longrightarrow \\ \beta}} X_\beta).
\]
\end{defi}
For a set $I$ of some morphisms to {\it permit the small object argument}, it is sufficient that the source domain of each morphism in $I$ is $\aleph_0$-small.
\begin{lem} \label{LemSmall}
For any compact Hausdorf space $X$, the condensed set $\overline{X}$ is $\aleph_0$-small.
\end{lem}
\begin{proof}
Take any regular cardinal $\lambda$ and any $\lambda$-sequence $Y$ in {\bf Cond}.
By Theorem \ref{main1}, each condensed set $Y_\beta$ is regarded as a sheaf on the site {\bf CH}.
By Yoneda's Lemma, we have the following bijection:
\[
{\bf Cond}(\overline{X}, Y_\beta) \cong Y_\beta(X).
\]
Because all coverings on the site {\bf CH} are finite, we have the following bijection:
\[
\displaystyle\lim_{\substack{\longrightarrow \\ \beta}} Y_\beta(X) \cong [\displaystyle\lim_{\substack{\longrightarrow \\ \beta}}Y_\beta](X).
\]
By Yoneda's Lemma again, we have the following bijection:
\[
{\bf Cond}(\overline{X}, \displaystyle\lim_{\substack{\longrightarrow \\ \beta}}Y_\beta) \cong [\displaystyle\lim_{\substack{\longrightarrow \\ \beta}}Y_\beta](X).
\]
We obtain the following bijection:
\[
\begin{array}{lcl}
\displaystyle\lim_{\substack{\longrightarrow \\ \beta}} {\bf Cond}(\overline{X}, Y_\beta)
& \cong & \displaystyle\lim_{\substack{\longrightarrow \\ \beta}} Y_\beta(X) \\
& \cong & [\displaystyle\lim_{\substack{\longrightarrow \\ \beta}}Y_\beta](X) \\
& \cong & {\bf Cond}(\overline{X}, \displaystyle\lim_{\substack{\longrightarrow \\ \beta}}Y_\beta).
\end{array}
\]
Then, the condensed set $\overline{X}$ is $\aleph_0$-small.
\end{proof}
	\subsection{Basic concepts and properties}
In this subsection, if we write $\Delta^n$, we shall refer to the simplicial set, and the simplex in {\bf Cond} shall be written as $F_0(\Delta_n)$.
Some simplicial complexes (e.g. $\partial \Delta^n$) shall be written in the same way.
Let $I_{\bf sSet}$ and $J_{\bf sSet}$ be the sets of some morphisms in {\bf sSet}, defined as the following:
\[
\begin{array}{lcl}
I_{\bf sSet} & = & \{ \partial \Delta^n \hookrightarrow \Delta^n \} \\
J_{\bf sSet} & = & \{ \Lambda^{n-1}_i \hookrightarrow \Delta^n \},
\end{array}
\]
where $\Lambda^{n-1}_i$ is a $(n-1, i)$-horn.
Remark that, for any morphism $f \in J_{\bf sSet}$, there exists a morphism $g \in J$ such that $g$ is isomorphic to $F_0(f)$. \par
We give a basic definition for a homotopy theory in {\bf Cond}, such as a {\it weak homotopy equivalence} and a {\it Serre fibration}.
\begin{defi}
A morphism $f$ in {\bf Cond} is a {\it weak homotopy equivalence} if the morphism $G_0(f)$ in {\bf sSet} is a weak equivalence.
A morphism $f$ in {\bf Cond} is a {\it Serre fibration} (resp. an {\it acyclic fibration}) if $f$ has the right lifting property with respect to the set $J$ (resp. $I$).
\end{defi}
The functor $F_0: {\bf sSet} \rightarrow {\bf Cond}$ preserves finite products.
\begin{lem} \label{LemProduct}
We have an isomorphism $F_0(X \times Y) \cong F_0(X) \times F_0(Y)$.
\end{lem}
\begin{proof}
First, we prove the case where $X$ and $Y$ are simplices.
$\Delta^n \times \Delta^m$ is a finite complex.
We have a coequalizer as the following:
\[
\coprod_{i, j} \Delta^{n+m-1}_{ij} \rightrightarrows \coprod_i \Delta^{n+m}_i \rightarrow \Delta^n \times \Delta^m,
\]
where all $\Delta^{n+m-1}_{ij}$ and $\Delta^{n+m}_i$ are simplices.
The functor $F_0$ preserves any colimit.
Then, the following diagram is a coequalizer:
\[
\coprod_{i, j} F_0(\Delta^{n+m-1}_{ij}) \rightrightarrows \coprod_i F_0(\Delta^{n+m}_i) \rightarrow F_0(\Delta^n \times \Delta^m).
\]
The functor $F_0$ preserves any simplex.
Then, the following diagram is also a coequalizer:
\[
\coprod_{i, j} F_0(\Delta^{n+m-1}_{ij}) \rightrightarrows \coprod_i F_0(\Delta^{n+m}_i) \rightarrow F_0(\Delta^n) \times F_0(\Delta^m).
\]
Therefore, we obtain an isomorphism $F_0(\Delta^n \times \Delta^m) \cong F_0(\Delta^n) \times F_0(\Delta^m)$. \par
Next, we prove the general case.
We have isomorphisms $X \cong \int^n {\bf sSet}(\Delta^{n-1}, X) \bullet \Delta^{n-1}$ and $Y \cong \int^n {\bf sSet}(\Delta^{n-1}, Y) \bullet \Delta^{n-1}$.
Then, we have the following isomorphisms:
\[
\begin{array}{lcl}
F_0(X \times Y)
& \cong & F_0\left( \left( \int^n {\bf sSet}(\Delta^{n-1}, X) \bullet \Delta^{n-1} \right) \times \left( \int^m {\bf sSet}(\Delta^{m-1}, Y) \bullet \Delta^{m-1} \right) \right) \\
& \cong & F_0\left(\int^n \int^m ({\bf sSet}(\Delta^{n-1}, X) \times {\bf sSet}(\Delta^{m-1}, Y)) \bullet (\Delta^{n-1} \times \Delta^{m-1}) \right) \\
& \cong & \int^n \int^m ({\bf sSet}(\Delta^{n-1}, X) \times {\bf sSet}(\Delta^{m-1}, Y)) \bullet F_0(\Delta^{n-1} \times \Delta^{m-1})) \\
& \cong & \int^n \int^m ({\bf sSet}(\Delta^{n-1}, X) \times {\bf sSet}(\Delta^{m-1}, Y)) \bullet (F_0(\Delta^{n-1}) \times F_0(\Delta^{m-1})) \\
& \cong & \left( \int^n {\bf sSet}(\Delta^{n-1}, X) \bullet F_0(\Delta^{n-1}) \right) \times \left( \int^m {\bf sSet}(\Delta^{m-1}, Y) \bullet F_0(\Delta^{m-1}) \right) \\
& \cong & F_0(X) \times F_0(Y).
\end{array}
\]
\end{proof}
Lemma \ref{LemProduct} implies the following corollary that the functor $G_0$ preserves the homotopic relation.
\begin{cor} \label{CorPath}
For any condensed set $X$, we have an isomorphism $G_0(X^{F_0(\Delta^1)}) \cong G_0(X)^{\Delta^1}$.
\end{cor}
\begin{proof}
For any simplicial set $Y$, we have the following bijection:
\[
\begin{array}{lcl}
{\bf sSet}(Y, G_0(X^{F_0(\Delta^1)}))
& \cong & {\bf Cond}(F_0(Y), X^{F_0(\Delta^1)}) \\
& \cong & {\bf Cond}(F_0(Y) \times F_0(\Delta^1), X) \\
& \cong & {\bf Cond}(F_0(Y \times \Delta^1), X) \\
& \cong & {\bf sSet}(Y \times \Delta^1, G_0(X)).
\end{array}
\]
Then, we obtain the isomorphism $G_0(X^{F_0(\Delta^1)}) \cong G_0(X)^{\Delta^1}$.
\end{proof}
The following lemma ensures that the counit of the derived adjunction of Quillen adjunction $F_0 \dashv G_0$ is a natural isomorphism.
\begin{lem} \label{LemCounit}
Let $\epsilon$ be the counit of the adjunction $F_0 \dashv G_0$.
Then, each component $\epsilon_X: F_0G_0X \rightarrow X$ is an acyclic fibration.
\end{lem}
\begin{proof}
For any square in the following diagram, we will show that there is a morphism $\gamma$ below.
	\[\xymatrix{
		F_0(\partial \Delta^n) \ar[r]^-{\alpha} \ar@{^(->}[d]
		& F_0(G_0(X)) \ar[d]^-{\epsilon_X}
	\\
		F_0(\Delta^n) \ar[r]_-{\beta} \ar@{.>}[ur]^-{\gamma}
		& X.
}\]
By definition, we obtain the following equality:
\[
\begin{array}{rcl}
F_0(G_0(X))
& = & \int^k {\bf sSet}(\Delta^{k-1}, G_0(X)) \bullet F_0(\Delta^{k-1}) \\
& \cong & \int^k G_0(X)_{k-1} \bullet F_0(\Delta^{k-1}) \\
& = & \int^k {\bf Cond}(F_0(\Delta^{k-1}), X) \bullet F_0(\Delta^{k-1}).
\end{array}
\]
Then, we have the quotient morphism $\coprod_k {\bf Cond}(F_0(\Delta^k), X) \bullet F_0(\Delta^k) \rightarrow F_0(G_0(X))$.
Let $\gamma$ be the composition of the following sequence:
\[
F_0(\Delta^n) \overset{\beta}{\rightarrow} {\bf Cond}(F_0(\Delta^n), X) \bullet F_0(\Delta^n) \rightarrow \coprod_k {\bf Cond}(F_0(\Delta^k), X) \bullet F_0(\Delta^k) \rightarrow F_0(G_0(X)).
\]
It can be verified that $\gamma$ makes the above diagram commutative.
Then, the morphism $\epsilon_X$ has the right lifting property with respect to $I$.
\end{proof}
	\subsection{HELP-lemma}
Gives a sufficient condition for a morphism in {\bf Cond} to be a weak homotopy equivalence, known as the {\it HELP-lemma} (cf. \cite{vogt2010help}).
\begin{defi}
A morphism $f$ in {\bf Cond} has the {\it (right) HELP (Homotopy-Extension-Lifting-Property)} (with respect to $I$) if, for any morphism $g \in I$ and any following diagram, there exsists a morphism $\gamma$ in the diagram such that it makes the upper triangle commutative and the lower triangle commutative up to homotopy.
\[\xymatrix{
		\partial \Delta^n \ar[d]_-{g} \ar[r]
		& \cdot \ar[d]^-{f}
	\\
		\Delta^n \ar[r] \ar@{.>}[ur]^-{\gamma}
		& \cdot
}\]
\end{defi}
\begin{lem} \label{LemCondHELP}
Let $f$ be a morphism in {\bf Cond}.
If $f$ has the HELP, then $G_0(f)$ is a weak equivalence in {\bf sSet}.
\end{lem}
\begin{proof}
Suppose that $f: X \rightarrow Y$ has the HELP.
Take a {\it mapping track} $f': X_f \rightarrow Y$ of $f$.
$f$ has a factorization $f = f' \circ e_f$ such that $e_f$ is a homotopy equivalence.
Then, $G_0(e_f)$ is a homotopy equivalence by Corollary \ref{CorPath}.
It can be checked that $f'$ has the right lifting property with respect to the set $I$ because $f$ has the HELP.
$G_0(f')$ is an acyclic fibration because of $F_0I_{\bf sSet} = I$.
Then, $G_0(f) = G_0(f') \circ G_0(e_f)$ is a weak equivalence.
\end{proof}
	\subsection{The model structure}
\begin{cor} \label{CorModel}
The category {\bf Cond} of the condensed sets has a model structure such that
\begin{itemize}
\item $f$ is a weak equivalence $\Leftrightarrow$ $f$ is a weak homotopy equivalence,
\item $f$ is a fibration $\Leftrightarrow$ $f$ is a Serre fibration, and
\item $f$ is a cofibration $\Leftrightarrow$ $f$ is a relative $I$-cell complex.
\end{itemize}
Moreover, $F_0 \dashv G_0$ and $F_1 \dashv G_1$ are Quillen equivalences.
\end{cor}
\begin{proof}
We will transfer the {\it cofibrantly generated} model structure on the category {\bf sSet} along the adjunction $F_0 \dashv G_0$.
To prove this, it is sufficient to prove the following two claims.
\begin{description}
\item[(1)] $I$ and $J$ permit the small object argument.
\item[(2)] Any relative $J$-cell complex is a weak homotopy equivalence.
\end{description}\par
The claim {\bf (1)} follows from Lemma \ref{LemSmall}.
We will show the claim {\bf (2)}.
Take any relative $J$-cell complex $f: X \rightarrow Y$.
Check that $f$ satisfies the HELP.
\[\xymatrix{
		\partial \Delta^n \ar[d] \ar[r]
		& X \ar[d]
	\\
		\Delta^n \ar[r]
		& Y
}\]
$f$ is the transfinite composition of a $\lambda$-sequence $Z$ such that each factor $Z_\beta \rightarrow Z_{\beta + 1}$ is the push-out of an element $j \in J$.
Let $\beta_0$ be the minimum ordinal of $\beta (\le \lambda)$ such that $Z_{\beta_0} \rightarrow Z_\lambda \cong Y$ has the HELP for the following diagram:
\[\xymatrix{
		\partial \Delta^n \ar[d] \ar[r]
		& X \ar[r]
		& Z_{\beta_0} \ar[d]
	\\
		\Delta^n \ar[rr]
		&& Y.
}\]
If we assume that $\beta_0$ is a limit ordinal, it contradicts that the functors ${\bf Cond}(\Delta^n, -)$ and ${\bf Cond}(\partial \Delta^n, -)$ preserve filtered colimits. 
If we assume that $\beta_0$ is a successor ordinal of an ordinal $\beta_0 - 1$, it contradicts that the morphism $Z_{\beta_0 - 1} \rightarrow Z_{\beta_0}$ is a strong deformation retract.
Then, we have the equality $\beta_0 = 0$.
Then, $f: X \cong Z_0 \rightarrow Y$ has the HELP.
This means that $f$ is a weak equivalence.
This proves the claim {\bf (2)}.
Then, the category {\bf Cond} has a model structure. \par
we will show that $F_0 \dashv G_0$ and $F_1 \dashv G_1$ are Quillen equivalences.
As is well known, $F_1F_0 \dashv G_0G_1$ is a Quillen equivalence.
It is only necessary to show that $F_0 \dashv G_0$ is a Quillen equivalence.
The unit of the derived adjunction of Quillen adjunction $F_0 \dashv G_0$ is a natural isomorphism because $F_1F_0 \dashv G_0G_1$ is a Quillen equivalence.
The counit of the derived adjunction of Quillen adjunction $F_0 \dashv G_0$ is a natural isomorphism because of Lemma \ref{LemCounit}.
Then, $F_0 \dashv G_0$ is a Quillen equivalence.
This implies that $F_1 \dashv G_1$ is also a Quillen equivalence.
\end{proof}

\bibliography{homotopy,category,condensed,h-principle,complex,misc}
\bibliographystyle{plain}


\end{document}